\documentclass{article}

\usepackage[french,british]{babel}

\usepackage[T1]{fontenc}

\usepackage{amsmath}

\usepackage{amsfonts}
\usepackage{amsbsy}
\usepackage{amssymb}

\usepackage{amsthm}

\theoremstyle{plain}

\usepackage{latexsym}
\usepackage{amsmath}
\usepackage{amsfonts}
\usepackage{amssymb}
\usepackage{psfrag}
\usepackage{graphicx}

\usepackage{amsmath,amssymb}
\usepackage{graphicx}

\usepackage{amsmath,amssymb}
\usepackage{graphicx}

\newtheorem{ozn}{Definition}[section]
\newtheorem{thm}{Theorem}[section]

\newtheorem{nas}{Corollary}[section]

\newtheorem{lema}{Lemma}[section]

\newcommand{\me}{\mathbf}
\newcommand{\mr}{\mathbb}
\newcommand{\mt}{\mathsf}
\newcommand{\md}{\mathcal}
\newcommand{\ld}{\left}
\newcommand{\rd}{\right}

\newcommand{\ip}{\int_{-\pi}^{\pi}}

\newcommand{\be}{\begin{equation}}
\newcommand{\ee}{\end{equation}}

\newcommand{\bem}{\begin{multline}}
\newcommand{\eem}{\end{multline}}

\newcommand{\bml}{\begin{multline*}}
\newcommand{\eml}{\end{multline*}}

\newcommand{\beg}{\begin{gather}}
\newcommand{\eeg}{\end{gather}}

\begin{document}

\title{Estimation problem for continuous time stochastic processes with periodically correlated increments}

\author{
Maksym Luz\thanks {BNP Paribas Cardif in Ukraine, Kyiv, Ukraine, maksym.luz@gmail.com},
Mikhail Moklyachuk\thanks
{Department of Probability Theory, Statistics and Actuarial
Mathematics, Taras Shevchenko National University of Kyiv, Kyiv 01601, Ukraine, moklyachuk@gmail.com}
}

\date{\today}

\maketitle

\renewcommand{\abstractname}{Abstract}
\begin{abstract}
We deal with the problem of optimal estimation of the linear   functionals
constructed from unobserved values of a continuous time  stochastic process with periodically correlated increments based on past observations of this process.
To solve the problem, we construct a corresponding to the process sequence of stochastic functions  which forms an infinite dimensional vector stationary increment sequence.
In the case of known spectral density of the stationary increment sequence, we obtain formulas for calculating values of the mean square errors and the spectral characteristics of the optimal estimates of the functionals.
Formulas determining  the least favorable spectral densities and the minimax (robust) spectral
characteristics of the optimal linear estimates of functionals
are derived in the case where the  sets of admissible spectral densities are given.
\end{abstract}

\maketitle

\textbf{Keywords}:{Periodically Correlated Increments, Minimax-Robust Estimate, Mean Square Error}

\maketitle

\vspace{2ex}
\textbf{\bf AMS 2010 subject classifications.} Primary: 60G10, 60G25, 60G35, Secondary: 62M20, 62P20, 93E10, 93E11

\section*{Introduction}

In this article, we study the prediction problem for the continuous time stochastic processes $\xi(t)$, $t\in\mr R$, with periodically
correlated increments $\xi^{(d)}(t,\tau T)=\Delta_{T\tau}^d \xi(t)$ of    order $d$ and period $T$, where $\Delta_{s} \xi(t)=\xi(t)-\xi(t-s)$. The resent studies, for example, by Basawa et al. \cite{Basawa}, Dudek et al. \cite{Dudek-Hurd}, Reisen et al. \cite{Reisen2018}, show a  constant interest to the non-stationary models and robust methods of estimation.

Kolmogorov \cite{Kolmogorov},  Wiener \cite{bookWien} and  Yaglom \cite{Yaglom} developed effective methods of solution  of interpolation, extrapolation (prediction) and filtering problems for stationary stochastic sequences and processes. For a particular problem, they found an estimate $\widetilde{x}(t)$ constructed from available observations that minimizes the mean square error $\Delta(\widetilde{x}(t),f)=\mt E|x(t)-\widetilde{x}(t)|^2$ in the case where the spectral density $f(\lambda)$ of the stationary process or sequence $x(t)$ is exactly known and fixed.
Such estimates  are called optimal linear estimates  within this article.

The developed  classical estimation methods are not directly applicable in practice since the exact  spectral structure of the processes is not usually available. In this case the estimated spectral densities can be considered as the true ones. However, Vastola and Poor \cite{VastPoor1983} showed with the help of the concrete examples, that such substitution  can result in a significant increase of  the   estimate error.  Therefore it is reasonable to consider the estimates, called minimax-robust, which minimize the maximum of the mean-square errors for all spectral densities from a given set of admissible spectral densities simultaneously. The minimax-robust method of extrapolation was proposed by Grenander \cite{Grenander} who considered an estimation of the functional $Ax = \int_{0}^1 a (t) x (t) dt$ as a  game between two players, one of which minimizes $\Delta(\widetilde{A}x,f)$ by $\widetilde{A}\zeta$ and another one maximizes it by $f$. The game has a saddle point under proper conditions:
\[ \max_{\substack{f}} \min_{\substack{\widetilde{A}x}} \mt \Delta(\widetilde{A}x,f) = \min_{\substack{\widetilde{A}x}} \max_{\substack {f}} \Delta(\widetilde{A}x,f) =\nu. \]
  For more details see the further study by Franke and Poor \cite{Franke1984} and  the survey paper by Kassam and Poor \cite{Kassam_Poor1985}. A wide range of results has been obtained by  Moklyachuk \cite{Moklyachuk:1981,Moklyachuk,Moklyachuk2015}. These results have been  extended on the vector-valued stationary processes and sequences by Moklyachuk and Masyutka \cite{Mokl_Mas_book}.

The concept of stationarity   admits some generalizations, a combination of two of which --   stationary $d$th increments and periodical  correlation  -- is in scope of this article.
Random processes with stationary $d$th increments $x(t)$ were introduced   by Yaglom and Pinsker \cite{Pinsker}. The increment sequence $x^{(d)}(t,\tau )=\Delta_{\tau}^d x(t)$ generated by such process is stationary by the variable $t$, namely, the mathematical expectations $\mt E\xi^{(n)}(t,\tau)$ and $\mt E\xi^{(n)}(t+s,\tau_1)\xi^{(n)}(t,\tau_2)$ do not depend on $t$.
Yaglom and Pinsker \cite{Pinsker} described the spectral representation of such process and the spectral density canonical factorization, and they also solved the extrapolation problem for these processes.
 The minimax-robust extrapolation, interpolation and filtering  problems for stochastic  processes with stationary increments were investigated by Luz and Moklyachuk \cite{Luz_Mokl_book}.

Dubovetska and Moklyachuk \cite{Dubovetska6}  derived the  classical and minimax-robust estimates for another generalization of stationary processes --  periodically  correlated (cyclostationary) processes, introduced by Gladyshev \cite{Glad1963}. The correlation function  $K(t,s)={\textsf{E}}{x(t)\overline{x(s)}}$ of such processes is a $T$-periodic function: $K(t,s)=K(t+T,s+T)$, which implies a time-dependent spectrum. Periodically  correlated  processes are widely used in signal processing and communications (see Napolitano \cite{Napolitano} for a review of recent works on cyclostationarity and its applications).

In this article, we deal with the problem of the mean-square optimal estimation of the linear functionals $A\xi=\int_{0}^{\infty}a(t)\xi(t)dt$ and $A_{NT}{\xi}=\int_{0}^{(N+1)T}a(t)\xi(t)dt$
which depend on the unobserved values of a continuous time stochastic process $\xi(t)$ with periodically stationary $d$-th
increments from its observations at points
$t<0$. The similar problems for discrete time processes have been studied by Kozak and Moklyachuk \cite{Kozak_Mokl}, Luz and Moklyachuk \cite{Luz_Mokl_extra_GMI,Luz_Mokl_extra_noise_PCI}.
   In section 2, we describe a presentation of a continuous time periodically stationary process as a  stationary H-valued sequence.
   This approach is extended on the periodically stationary increments in the subsection 3.1.
   In   subsection 3.2, the maximum mean square error of the estimate is derived. The classical prediction (or extrapolation) problem is solved in subsection 3.3.
   Particularly, formulas for calculating  the mean-square errors and the spectral characteristics of the optimal linear
estimates of the functionals $A\xi$ and $A_{NT}{\xi}$ are derived under the condition of spectral certainty.
In subsection 4, we present our results on minimax-robust prediction for the studied processes: relations that determine the least favourable spectral densities and the minimax spectral
characteristics derived for some classes of spectral densities.

\section{Continuous time periodically correlated processes and generated vector stationary sequences}
\label{section_PC_process}
In this section we present a brief review of periodically correlated processes and describe an approach to be applied in the next section to develop a spectral theory for periodically correlated increment process.

\begin{ozn}[Gladyshev \cite{Glad1963}] A mean-square continuous stochastic process
$\zeta:\mathbb R\to H=L_2(\Omega,\mathcal F,\mathbb P)$, with $\textsf {E} \zeta(t)=0,$ is called periodically correlated (PC) with period  $T$, if its correlation function  $K(t,s)={\textsf{E}}{\zeta(t)\overline{\zeta(s)}}$  for all  $t,s\in\mathbb R$ and some fixed $T>0$ is such that
\[
K(t,s)={\textsf{E}}{\zeta(t)\overline{\zeta(s)}}={\textsf{E}}{\zeta(t+T)\overline{\zeta(s+T)}}=K(t+T,s+T).
\]
\end{ozn}

For a periodically correlated  stochastic process $\zeta(t)$, one can construct
the following sequence of stochastic functions \cite{DubovetskaMoklyachuk2013}, \cite{MoklyachukGolichenko2016}
\begin{equation} \label{zj}
\{\zeta_j(u)=\zeta(u+jT),u\in [0,T), j\in\mathbb Z\}.
\end{equation}
Sequence (\ref{zj}) forms a $L_2([0,T);H)$-valued
stationary sequence $\{\zeta_j,j\in\mathbb Z\}$ with the correlation function
\[
B_{\zeta}(l,j) = \langle\zeta_l,\zeta_j\rangle_H =\int_0^T
\textsf{E}[\zeta(u+lT)\overline{\zeta(u+jT)}]du
=\int_0^TK_{\zeta}(u+(l-j)T,u)du =
 B_{\zeta}(l-j),
 \]
where
$K_{\zeta}(t,s)=\textsf{E}{\zeta}(t)\overline{{\zeta}(s)}$
is the correlation function of the PC process $\zeta(t)$.
If we chose the following orthonormal basis in the space $L_2([0,T);\mathbb{R})$
\be\label{orthonormal_basis}
\{\widetilde{e}_k=\frac{1}{\sqrt{T}}e^{2\pi
i\{(-1)^k\left[\frac{k}{2}\right]\}u/T}, k=1,2,3,\dots\}, \quad
\langle \widetilde{e}_j,\widetilde{e}_k\rangle=\delta_{kj},
\ee
 the stationary sequence $\{\zeta_j,j\in\mathbb Z\}$  can be represented in the form
\[
\zeta_j= \sum_{k=1}^\infty \zeta_{kj}\widetilde{e}_k,
\]
where
\[\zeta_{kj}=\langle\zeta_j,\widetilde{e}_k\rangle =
\frac{1}{\sqrt{T}} \int_0^T \zeta_j(v)e^{-2\pi
i\{(-1)^k\left[\frac{k}{2}\right]\}v/T}dv.\]
The sequence
$\{\zeta_j,j\in\mathbb Z\},$
   or the corresponding to it a vector sequence
  \[\{\vec \zeta_j=(\zeta_{kj}, k=1,2,\dots)^{\top},
j\in\mathbb Z\},\]
is called a generated by the process $\{\zeta(t),t\in\mathbb R\}$ vector stationary sequence.
Components $\{\zeta_{kj}\}: k=1,2,\dots;j\in\mathbb Z$  of the generated stationary sequence
$\{\zeta_j,j\in\mathbb Z\}$ satisfy the following relations \cite{Kallianpur}, \cite{Moklyachuk:1981}
\[
\textsf{E}{\zeta_{kj}}=0, \quad \|{\zeta}_j\|^2_H=\sum_{k=1}^\infty
\textsf{E}|\zeta_{kj}|^2\leq P_\zeta=B_\zeta(0), \quad
\textsf{E}\zeta_{kl}\overline{\zeta}_{nj}=\langle
R_{\zeta}(l-j)\widetilde{e}_k,\widetilde{e}_n\rangle.
\]
The correlation function $R_{\zeta}(j)$  of the generated stationary
sequence $\{\zeta_j,j\in\mathbb Z\}$
 are correlation operator functions.
 The correlation operator $R_{\zeta}(0)=R_{\zeta}$ is a
kernel operator and its kernel norm satisfies the following limitations:
\[
\|{\zeta}_j\|^2_H=\sum_{k=1}^\infty \langle R_{\zeta}z
\widetilde{e}_k,\widetilde{e}_k\rangle\leq P_\zeta. \]
The generated stationary sequence $\{\zeta_j,j\in\mathbb Z\}$ has the spectral density function
$f(\lambda)=\{f_{kn}(\lambda)\}_{k,n=1}^\infty$, that is positive valued operator  functions of variable
 $\lambda\in [-\pi,\pi)$, if its correlation function $R_{\zeta}(j)$ can be represented in the form
\[
\langle R_{\zeta}(j)\widetilde{e}_k,\widetilde{e}_n\rangle=\frac{1}{2\pi} \int _{-\pi}^{\pi}
e^{ij\lambda}\langle f(\lambda) \widetilde{e}_k,\widetilde{e}_n\rangle d\lambda.
\]
We finish our review by the statement, that
For almost all   $\lambda\in [-\pi,\pi)$ the spectral density $f(\lambda)$ is a kernel operator with an integrable kernel norm
\[
\sum_{k=1}^\infty \frac{1}{2\pi} \int _{-\pi}^\pi \langle f(\lambda)
\widetilde{e}_k,\widetilde{e}_k\rangle d\lambda=\sum_{k=1}^\infty\langle R_{\zeta}
\widetilde{e}_k,\widetilde{e}_k\rangle=\|{\zeta}_j\|^2_H\leq P_\zeta.
\]

\section{Extrapolation problem for stochastic processes with periodically correlated $d$th increments}

\subsection{Stochastic processes with periodically correlated $d$th increments}

For a given stochastic process $\{\xi(t),t\in\mathbb R\}$, consider the stochastic $d$th increment process
\begin{equation}
\label{oznachPryrostu_cont}
\xi^{(d)}(t,\tau)=(1-B_{\tau})^d\xi(t)=\sum_{l=0}^d(-1)^l{d \choose l}\xi(t-l\tau),
\end{equation}
with the step $\tau\in\mr R$, generated by the stochastic process $\xi(t)$. Here $B_{\tau}$ is the backward shift operator: $B_{\tau}\xi(t)=\xi(t-\tau)$, $\tau\in \mr R$.

We prefer to use the notation $\xi^{(d)}(t,\tau)$ instead of widely used $\Delta_{\tau}^{d}\xi(t)$ to avoid a duplicate with the mean square error notation.

\begin{ozn}
\label{OznPeriodProc2_cont}
A stochastic process $\{\xi(t),t\in\mathbb R\}$ is called a \emph{stochastic
process with periodically stationary (periodically correlated) increments} with the step $\tau\in\mr Z$ and the period $T>0$ if the mathematical expectations
\begin{eqnarray*}
\mt E\xi^{(d)}(t+T,\tau T) & = & \mt E\xi^{(d)}(t,\tau T)=c^{(d)}(t,\tau T),
\\
\mt E\xi^{(d)}(t+T,\tau_1 T)\xi^{(d)}(s+T,\tau_2 T)
& = & D^{(d)}(t+T,s+T;\tau_1T,v_2T)
\\
& = & D^{(d)}(t,s;\tau_1T,\tau_2T)
\end{eqnarray*}
exist for every  $t,s\in \mr R$, $\tau_1,\tau_2 \in \mr Z$ and for some fixed $T>0$.
\end{ozn}

The functions $c^{(d)}(t,\tau T)$ and  $D^{(d)}(t,s;\tau_1T,\tau_2 T)$ from the Definition \ref{OznPeriodProc2_cont} are called the \emph{mean value} and  the \emph{structural function} of the stochastic
process $\xi(t)$ with periodically stationary (periodically correlated) increments.

For the stochastic process $\{\xi(t), t\in \mathbb R\}$ with periodically correlated  increments $\xi^{(d)}(t,\tau T)$ and the integer step $\tau$, we follow the procedure described in the Section \ref{section_PC_process} and construct
a sequence of stochastic functions
\begin{equation} \label{xj}
\{\xi^{(d)}_j(u):=\xi^{(d)}_{j,\tau}(u)=\xi^{(d)}_j(u+jT,\tau T),\,\, u\in [0,T), j\in\mathbb Z\}.
\end{equation}

Sequence (\ref{xj}) forms a $L_2([0,T);H)$-valued
stationary increment sequence $\{\xi^{(d)}_j,j\in\mathbb Z\}$ with the structural function
\begin{eqnarray*}
B_{\xi^{(d)}}(l,j)&= &\langle\xi^{(d)}_l,\xi^{(d)}_j\rangle_H =\int_0^T
\textsf{E}[\xi^{(d)}_j(u+lT,\tau_1 T)\overline{\xi^{(d)}_j(u+jT,\tau_2 T)}]du
\\&=&\int_0^T D^{(d)}(u+(l-j)T,u;\tau_1T,\tau_2T) du =
 B_{\xi^{(d)}}(l-j).
 \end{eqnarray*}
Making use of the orthonormal basis \eqref{orthonormal_basis} the stationary increment sequence  $\{\xi^{(d)}_j,j\in\mathbb Z\}$ can be represented in the form
\begin{equation} \label{zeta}
\xi^{(d)}_j= \sum_{k=1}^\infty \xi^{(d)}_{kj}\widetilde{e}_k,\end{equation}
where
\[\xi^{(d)}_{kj}=\langle\xi^{(d)}_j,\widetilde{e}_k\rangle =
\frac{1}{\sqrt{T}} \int_0^T \xi^{(d)}_j(v)e^{-2\pi
i\{(-1)^k\left[\frac{k}{2}\right]\}v/T}dv.\]

We call this sequence
$\{\xi^{(d)}_j,j\in\mathbb Z\},$
   or the corresponding to it vector sequence
  \[\{\vec\xi^{(d)}_j=(\xi^{(d)}_{kj}, k=1,2,\dots)^{\top},
j\in\mathbb Z\},\]
\emph{an infinite dimension vector stationary increment sequence} generated by the increment process $\{\xi^{(d)}(t,\tau T),t\in\mathbb R\}$.
Further, we will omit the word vector in the notion generated vector stationary increment sequence.

Components $\{\xi^{(d)}_{kj}\}: k=1,2,\dots;j\in\mathbb Z$  of the generated stationary increment sequence
$\{\xi^{(d)}_j,j\in\mathbb Z\}$ are such that, \cite{Kallianpur}, \cite{Moklyachuk:1981}
\[
\textsf{E}{\xi^{(d)}_{kj}}=0, \quad \|\xi^{(d)}_j\|^2_H=\sum_{k=1}^\infty
\textsf{E}|\xi^{(d)}_{kj}|^2\leq P_{\xi^{(d)}}=B_{\xi^{(d)}}(0),
\]
and
\[\textsf{E}\xi^{(d)}_{kl}\overline{\xi^{(d)}_{nj}}=\langle
R_{\xi^{(d)}}(l-j;\tau_1,\tau_2)\widetilde{e}_k,\widetilde{e}_n\rangle.
\]
The \emph{structural function}  $R_{\xi^{(d)}}(j):=R_{\xi^{(d)}}(j;\tau_1,\tau_2)$  of the generated stationary increment
sequence $\{\xi^{(d)}_j,j\in\mathbb Z\}$
 is a correlation operator function.
 The correlation operator $R_{\xi^{(d)}}(0)=R_{\xi^{(d)}}$ is a
kernel operator and its kernel norm satisfies the following limitations:
\[
\|\xi^{(d)}_j\|^2_H=\sum_{k=1}^\infty \langle R_{\xi^{(d)}}
\widetilde{e}_k,\widetilde{e}_k\rangle\leq P_{\xi^{(d)}}. \]

 Suppose that  the structural function $R_{\xi^{(d)}}(j)$ admits a representation
\[
\langle R_{\xi^{(d)}}(j;\tau_1, \tau_2)\widetilde{e}_k,\widetilde{e}_n\rangle=\frac{1}{2\pi} \int _{-\pi}^{\pi}
e^{ij\lambda}(1-e^{-i\tau_1\lambda})^d(1-e^{i\tau_2\lambda})^d\frac{1}
{\lambda^{2d}}\langle f(\lambda) \widetilde{e}_k,\widetilde{e}_n\rangle d\lambda.
\]
Then
$f(\lambda)=\{f_{kn}(\lambda)\}_{k,n=1}^\infty$ is a \emph{spectral density function} of the generated stationary increment sequence $\{\xi^{(d)}_j,j\in\mathbb Z\}$. It  is a positive valued operator  functions of variable
 $\lambda\in [-\pi,\pi)$, and for almost all   $\lambda\in [-\pi,\pi)$ it is a kernel operator with an integrable kernel norm
\begin{equation} \label{P1}
\sum_{k=1}^\infty \frac{1}{2\pi} \int _{-\pi}^\pi (1-e^{-i\tau_1\lambda})^d(1-e^{i\tau_2\lambda})^d\frac{1}
{\lambda^{2d}} \langle f(\lambda)
\widetilde{e}_k,\widetilde{e}_k\rangle d\lambda=\sum_{k=1}^\infty\langle R_{\xi^{(d)}}
\widetilde{e}_k,\widetilde{e}_k\rangle=\|{\zeta}_j\|^2_H\leq P_{\xi^{(d)}}.
\end{equation}

\subsection{Extrapolation problem: the greatest value of the mean-square error}

Consider the problem of the mean square optimal linear  estimation  of the functionals
\[A {\xi}=\int_{0}^{\infty}a(t)\xi(t)dt,\quad A_{NT}{\xi}=\int_{0}^{(N+1)T}a(t)\xi(t)dt\]
which depend on the unknown values of the stochastic process $\xi(t)$ with periodically correlated $d$th increments.
Estimates are based on observations of the process $\xi(t)$  at points $t<0$.

 \begin{lema}[Luz and Moklyachuk \cite{Luz_Mokl_book}]\label{lema predst A_cont}
Any linear functional
\[
    A\xi=\int_{0}^{\infty}a(t)\xi(t)dt\]
allows the representation
\[
    A\xi=B\xi-V\xi,\]
where
\[
    B\xi=\int_{0}^{\infty}b^{\tau}(t)\xi^{(d)}(t,\tau T)dt,\quad V\xi=\int_{-\tau T d}^{0}v^{\tau}(t)\xi(t)dt,\]
and
\begin{eqnarray} \label{koefv_cont}
    v^{\tau}(t)&=&\sum_{l=\left[-\frac{t}{\tau T}\right]'}^d(-1)^l{d \choose l}b^{\tau}(t+l\tau T),\quad t\in [-\tau T d;0),
\\
    \label{koef b_cont}b^{\tau}(t)&=&\sum_{k=0}^{\infty}a(t+\tau T k)d(k)=D^{\tau T}\me a(t) ,\,\,t\geq0,
\end{eqnarray}
 where  $[x]'$ denotes the least integer number among numbers that are equal to or greater than  $x$, coefficients
  $\{d(k):k\geq0\}$ are determined by the relation
\[\sum_{k=0}^{\infty}d (k)x^k=\left(\sum_{j=0}^{\infty}x^{
j}\right)^d,\]
$D^{\tau T}$ is the linear transformation acting on an arbitrary function $x(t)$, $t\geq0$, as follows:
\[
    D^{\tau T}\me x(t)=\sum_{k=0}^{\infty}x(t+\tau T k )d(k).\]
\end{lema}

\begin{nas}\label{nas predst A_T_cont}
A linear functional
\[A_{NT}\xi=\int_{0}^{(N+1)T}a(t)\xi(t)dt\]
allows the representation
\[
    A_{NT}\xi=B_{NT}\xi-V_{NT}\xi,\]
where
\[
    B_{NT}\xi=\int_0^{(N+1)T} b^{\tau,N}(t)\xi^{(d)}(t,\tau T)dt,\quad V_{NT}\xi=\int_{-\tau T d}^{0}v^{\tau,N}(t)\xi(t)dt,\]
and
\begin{eqnarray}\label{koefv_N_cont}
    v^{\tau,N}(t)&=&\sum_{l=\left [-\frac{t}{\tau T}\right ]'}^{\min\left \{\left [\frac{{(N+1)T}-t}{\tau T}\right ],d\right \}}(-1)^l{d \choose l}b^{\tau, N}(t+l\tau T),\quad t\in[-\tau T d;0),
\\
\label{koef_N b_cont}
    b^{\tau,N}(t)&=&\sum_{k=0}^{\left[\frac{{(N+1)T}-t}{\tau T}\right]}a( t +\tau T k)d(k)= D^{\tau T,N}\me a(t),\,
t\in[0;{(N+1)T}],
\end{eqnarray}
the linear transformation $D^{\tau T,N}$ acts on an arbitrary function $x(t)$, $t\in[0;{(N+1)T}]$, as follows:
\[
    D^{\tau T,N}\me x(t)=\sum_{k=0}^{\left[\frac{{(N+1)T}-t}{\tau T}\right]}x(t+\tau T k)d(k).\]
\end{nas}

Let $\widehat{A}\xi$ denote the mean square optimal linear estimate of the functional ${A}\xi$ from observations of the
process $\xi(t)$ at points $ t < 0$ and let $\widehat{B}\xi$ denote the mean
square optimal linear estimate of the functional ${B}\xi$ from
observations of the stochastic dth increment process
$\xi^{(d)}(t,\tau T)$ at points $ t < 0$.

Let
$
\Delta(f;\widehat{A}\xi)=\mt E |A\xi-\widehat{A}\xi|^2
$ and
$
\Delta(f;\widehat{B}\xi)=\mt E |B\xi-\widehat{B}\xi|^2
$
denote the mean square errors of the estimates $\widehat{A}\xi$ and $\widehat{B}\xi$ respectively/\.
Since values $\xi(t)$ at points $t \in [-\tau T n; 0)$ are
known, the following equality comes from Lemma \ref{lema predst A_cont}:
\be\label{mainformula}
    A\xi=B\xi-V\xi,\quad \widehat{A}\xi=\widehat{B}\xi-V\zeta,\ee
and
\[
    \Delta\ld(f;\widehat{A}\xi\rd)=\mt E \ld|A\xi-\widehat{A}\xi\rd|^2= \mt
    E\ld|B\xi-V\xi-\widehat{B}\xi+V\xi\rd|^2
    =\mt E\ld|B\xi-\widehat{B}\xi\rd|^2=\Delta\ld(f;\widehat{B}\xi\rd).
   \]
The  similar relations for the functional $A_{NT}\xi$ are valid:
 \be\label{mainformulaANT}
    A_{NT}\xi=B_{NT}\xi-V_{NT}\xi, \quad \widehat{A}_{NT}\xi=\widehat{B}_{NT}\xi-V_{NT}\zeta,\ee
and
\begin{eqnarray*}
    \Delta\ld(f;\widehat{A}_{NT}\xi\rd)&=&\mt E \ld|A_{NT}\xi-\widehat{A}_{NT}\xi\rd|^2= \mt
    E\ld|B_{NT}\xi-V_{NT}\xi-\widehat{B}_{NT}\xi+V_{NT}\xi\rd|^2
    \\&=&\mt E\ld|B_{NT}\xi-\widehat{B}_{NT}\xi\rd|^2=\Delta\ld(f;\widehat{B}_{NT}\xi\rd).
    \end{eqnarray*}

Thus, to find  the mean square optimal linear  estimation  of the functionals $A\xi$, $A_{NT}\xi$
 we have to find estimates of the functionals $B\xi$, $B_{NT}\xi$.

The functional $B_{NT}\xi$ can be represented in the form
 \[
    B_{NT}\xi=\int_0^{(N+1)T} b^{\tau,N}(t)\xi^{(d)}(t,\tau T)dt
= \sum_{j=0}^{N}\int_{0}^{T}
b^{\tau,N}_j(u)\xi^{(d)}_j(u)
du,
\]
where
\[
b^{\tau,N}_j(u)=b^{\tau,N}(u+jT),\quad \xi^{(d)}_j(u)=\xi^{(d)}_j(u+jT,\tau T),\,\, u\in [0,T).
\]

Making use of the decomposition (\ref{zeta}) of the generated stationary increment sequence  $\{\xi^{(d)}_j,j\in\mathbb Z\}$  and the solution of equation
\begin{equation} \label{peretv}
(-1)^k\left[\frac{k}{2}\right]+(-1)^m\left[\frac{m}{2}\right]=0
\end{equation}
of two variables  $(k,m)$, which is given by pairs  $(1,1)$, $(2l+1,2l)$ and
$(2l,2l+1)$ for $l=2,3,\dots$, the functional $B_{NT}\xi$ can be rewritten in the form
\cite{MoklyachukGolichenko2016}
\begin{eqnarray*}
    B_{NT}\xi&=&\int_0^{(N+1)T} b^{\tau,N}(t)\xi^{(d)}(t,\tau T)dt
= \sum_{j=0}^{N}\int_{0}^{T}
b^{\tau,N}_j(u)\xi^{(d)}_j(u)
du
\\
&=&\sum_{j=0}^N\frac{1}{T}\int_{0}^{T}\left(\sum_{k=1}^\infty b^{\tau,N}_{kj}
e^{2\pi i\{(-1)^k\left[\frac{k}{2}\right]\}u/T} \right)
\left(\sum_{m=1}^\infty \xi^{(d)}_{mj} e^{2\pi
i\{(-1)^m\left[\frac{m}{2}\right]\}u/T} \right)du
\\
&=&\sum_{j=0}^N\sum_{k=1}^\infty \sum_{m=1}^\infty b^{\tau,N}_{kj} \xi^{(d)}_{mj}
\frac{1}{T}\int_{0}^{T}e^{2\pi
i\left\{(-1)^k\left[\frac{k}{2}\right]+(-1)^m\left[\frac{m}{2}\right]\right\}u/T}
du
\\
&=&\sum_{j=0}^N\sum_{k=1}^{\infty} b^{\tau,N}_{kj}\xi^{(d)}_{kj}=
 \sum_{j=0}^N
{(\vec{b}^{\tau,N}_j)}^{\top}\vec{\xi}^{(d)}_j=B_{N}\vec \xi,
\end{eqnarray*}
where the infinite dimensional vectors $\vec{\xi}^{(d)}_j$ and $\vec{b}^{\tau,N}_j$ are defined as follows:
\[
\vec{\xi}^{(d)}_j=(\xi^{(d)}_{kj},k=1,2,\dots)^{\top},
\]
and
\[
\vec{b}^{\tau,N}_j=(b^{\tau,N}_{kj},k=1,2,\dots)^{\top}=
(b^{\tau,N}_{1j},b^{\tau,N}_{3j},b^{\tau,N}_{2j},\dots,b^{\tau,N}_{2k+1,j},b^{\tau,N}_{2k,j},\dots)^{\top}.
\]
Here
\[
 b^{\tau,N}_{kj}=\langle b^{\tau,N}_j,\widetilde{e}_k\rangle =
\frac{1}{\sqrt{T}} \int_0^T b^{\tau,N}_j(v)e^{-2\pi
i\{(-1)^k\left[\frac{k}{2}\right]\}v/T}dv,
\]
where
\[
b^{\tau,N}_j(u)=\sum_{l=0}^{\left[\frac{{(N+1)T}-u-jT}{\tau T}\right]}a( u+jT +\tau T l)d(l)= D^{\tau T,N} a(u),\,
u\in[0;t), \, j=0,1,\ldots,N,
\]
and
\[
b^{\tau,N}_{kj}=\sum_{k=0}^{\left[\frac{N-j}{\tau}\right]}a_{kj+\tau l}d(k),\,
 j=0,1,\ldots,N,
\]
where
\[
 a_{kj}=\langle a_j,\widetilde{e}_k\rangle =
\frac{1}{\sqrt{T}} \int_0^T a_j(v)e^{-2\pi
i\{(-1)^k\left[\frac{k}{2}\right]\}v/T}dv,\quad k=1,2,\dots,\,j=0,1,\ldots,N.
\]
Assume, that coefficients  $\{\vec{b}^{\tau,N}_j, j=0,1,\dots,N\}$, that determine the functional $B_{NT}\xi$, satisfy the conditions
\begin{equation} \label{an}
 \|\vec{b}^{\tau,N}_j\|<\infty, \quad \|\vec{b}^{\tau,N}_j\|^2=\sum_{k=1}^\infty |{b}^{\tau,N}_{kj}|^2,\quad    j=0,1,\dots,N.
\end{equation}
It follows from the condition (\ref{an}) that the functional
\[B_{NT}\xi=
\sum_{j=0}^N\sum_{k=1}^{\infty} b^{\tau,N}_{kj}\xi^{(d)}_{kj}=
 \sum_{j=0}^N
{(\vec{b}^{\tau,N}_j)}^{\top}\vec{\xi}^{\,(d)}_j=B_{N}\vec \xi
\]
has a finite second moment.

Let us calculate the greatest value of the mean-square error  $\Delta(\xi,\hat{B}_{NT})$ of estimate  $\hat B_{NT}\xi$ of the functional $B_{NT}\xi$.
Denote by $\Lambda$  the set of all linear estimates of the functional $B_{NT}\xi$,  based on an observation of the process  $\xi(t)$
at points   $t<0$.

Let  $\mathbf{Y}$ denote the class of the mean-square continuous PC increment processes $\xi^{(d)}(t, \tau T)$,   such that
$\textsf{E}\xi^{(d)}(t, \tau T)=0, \textsf{E} |\xi^{(d)}(t, \tau T)|^2\leq P_{\xi^{d}}$.
Let  $\mathbf{Y}_R$ denote the class of all regular stationary sequences, that satisfy the condition  $\|\xi^{(d)}_j\|^2_H\leq
P_{\xi^{(d)}}$. Then we can formulate the following theorem (see \cite{Moklyachuk:1981,Moklyachuk:2008}).

\begin{thm} \label{111}
Let the coefficients $\{\vec{b}^{\tau,N}_j, j=0,1,\dots,N\}$,  which determine the functional $B_{NT}\xi$,  satisfy condition (\ref{an}). The function $\Delta(\xi,\hat{B}_{NT})$  has a saddle point on the set $\mathbf{Y} \times \Lambda$:
\[
 \mathop {\min }\limits_{\hat B_{NT} \in \Lambda} \mathop
{\max }\limits_{\zeta \in \mathbf{Y}} \Delta(\xi,\hat{B}_{NT})=
\mathop {\max }\limits_{\zeta \in \mathbf{Y}} \mathop {\min}\limits_{\hat B_{NT} \in \Lambda}
\Delta(\xi,\hat{B}_{NT})
=P_\xi
\cdot \nu_{N}^2,
\]
where $\nu_{N}^2$ is the greatest eigenvalue of the self-adjoint compact operator  $Q_{N}$ in the space $\ell_2$,  determined  by the matrix $\{Q_N(p,q)\}_{p,q=0}^D$, which is constructed from the block-matrices $Q_N(p,q)=\{Q_{kn}^D(p,q)\}_{k,n=1}^\infty$   with elements
\begin{equation} \label{qn}
Q_{kn}^{N}(p,q)=\sum_{s=0}^{min(N-p,N-q)}b^{\tau,N}_{k,s+p}\cdot
\overline{b^{\tau,N}_{n,s+q}},
\end{equation}
\[ k,n=1,2,\dots, \quad p,q=0,1,\dots,N.\]
The least favorable stochastic sequence $\vec{\xi}^{(d)}_j$ in the class $\mathbf{Y}$  for the optimal estimate of the functional  $B_{NT}\xi$  is an one-sided moving average sequence of order $N$ of the form
\[
\vec{\xi}^{(d)}_j= \sum_{s=j-N}^j
 g(j-s) \vec \varepsilon(s),\]
where $g_N=(g(0),g(1),\dots,g(N))^{\top}$ is the eigenvector, that corresponds to  $\nu_N^2$, which is constructed from matrices  $g(p)=\{g_{km}(p)\}_{k=\overline{1,\infty}}^{m=\overline{1,M}}$  and is determined by the condition $\|g_N\|^2=\sum _{p=0}^D
 \|g(p)\|^2=P_\zeta$,
$\vec  \varepsilon(s)=\{\varepsilon_{m}(s)\}_{m=1}^M$ is a vector valued stationary stochastic sequence with orthogonal values.
\end{thm}

The similar representations can be obtained for the   functional $B\xi$   \cite{MoklyachukGolichenko2016}:
 \[
 B\xi=\int_{0}^{\infty}b^{\tau}(t)\xi^{(d)}(t,\tau T)dt=
 \sum_{j=0}^{\infty}
{(\vec{b}^{\tau}_j)}^{\top}\vec{\xi}^{(d)}_j=B\vec\xi,
\]
where the infinite dimensional vectors $\vec{\xi}^{(d)}_j$ and $\vec{b}^{\tau}_j$ are defined as follows:
\begin{eqnarray*}
\vec{\xi}^{(d)}_j&=&(\xi^{(d)}_{kj},k=1,2,\dots)^{\top},
\\
\vec{b}^{\tau}_j&=&(b^{\tau}_{kj},k=1,2,\dots)^{\top}=
(b^{\tau}_{1j},b^{\tau}_{3j},b^{\tau}_{2j},\dots,b^{\tau}_{2k+1,j},b^{\tau}_{2k,j},\dots)^{\top}.
\end{eqnarray*}
Here
\begin{eqnarray*}
 b^{\tau}_{kj}&=&\langle b^{\tau}_j,\widetilde{e}_k\rangle =
\frac{1}{\sqrt{T}} \int_0^T b^{\tau}_j(v)e^{-2\pi
i\{(-1)^k\left[\frac{k}{2}\right]\}v/T}dv,
\\
   b^{\tau}_j(u)&=&\sum_{l=0}^{\infty}a(u+jT+\tau T l)d(l)=D^{\tau} a_j(u) ,\quad u\in [0,T), \, j=0,1,\ldots,
\end{eqnarray*}
and
\begin{eqnarray*}
    b^{\tau}_{k,j}&=&\sum_{l=0}^{\infty}a_{k,j+\tau l}d(l),\quad k,j=0,1,\ldots,
\\
 a_{kj}&=&\langle a_j,\widetilde{e}_k\rangle =
\frac{1}{\sqrt{T}} \int_0^T a_j(v)e^{-2\pi
i\{(-1)^k\left[\frac{k}{2}\right]\}v/T}dv.
\end{eqnarray*}

Assume, that coefficients  $\{\vec{b}^{\tau}_j, j=0,1,\dots\}$, that determine the functional $B\xi$, satisfy the conditions
\begin{equation} \label{ant}
 \sum_{j=0}^\infty\|
 \vec{b}^{\tau}_j\|<\infty, \quad
 \sum_{j=0}^\infty(j+1)\|
 \vec{b}^{\tau}_j\|<\infty, \quad
  \|\vec{b}^{\tau}_j\|^2=\sum_{k=1}^\infty |{b}^{\tau}_{kj}|^2.
\end{equation}
It follows from the condition (\ref{ant}) that the functional
\[B\xi=
\sum_{j=0}^{\infty}\sum_{k=1}^{\infty} b^{\tau}_{kj}\xi^{(d)}_{kj}=
 \sum_{j=0}^{\infty}
{(\vec{b}^{\tau}_j)}^{\top}\vec{\xi}^{\,(n)}_j
\]
has a finite second moment. The following theorem provides the greatest value of the mean-square error  $\Delta(\xi,\hat{B})$ of the estimate  $\hat B\xi$ of the functional $B\xi$ (see \cite{Moklyachuk:1981,Moklyachuk:2008}).

\begin{thm} \label{11122}
Let the coefficients $\{\vec{b}^{\tau}_j, j=0,1,\dots\}$,  which determine the functional $B\xi$,  satisfy condition (\ref{ant}). The function $\Delta(\xi,\hat{B})$  has a saddle point on the set $\mathbf{Y} \times \Lambda$:
\[
 \mathop {\min }\limits_{\hat B \in \Lambda} \mathop
{\max }\limits_{\zeta \in \mathbf{Y}} \Delta(\xi,\hat{B})=
\mathop {\max }\limits_{\zeta \in \mathbf{Y}} \mathop {\min}\limits_{\hat B\in \Lambda}
\Delta(\xi,\hat{B})
=P_\xi
\cdot \nu^2,
\]
where $\nu^2$ is the greatest eigenvalue of the self-adjoint compact operator  $Q$ in the space $\ell_2$,  determined  by the matrix $\{Q(p,q)\}_{p,q=0}^{\infty}$, which is constructed from the block-matrices $Q(p,q)=\{Q_{kn}(p,q)\}_{k,n=1}^\infty$   with elements
\begin{equation} \label{qn}
Q_{kn}(p,q)=\sum_{s=0}^{\infty}b^{\tau}_{k,s+p}\cdot\overline{b^{\tau}_{n,s+q}},
\end{equation}
\[ k,n=1,2,\dots, \quad p,q=0,1,\dots\]
The least favorable stochastic sequence $\vec{\xi}^{(d)}_j$ in the class $\mathbf{Y}$  for the optimal estimate of the functional  $B\xi$  is an one-sided moving average sequence of the form
\[
\vec{\xi}^{(d)}_j= \sum_{s=-\infty}^j
 g(j-s) \vec \varepsilon(s),\]
where $g=(g(0),g(1),\dots)^{\top}$ is the eigenvector, that corresponds to  $\nu^2$, which is constructed from matrices  $g(p)=\{g_{km}(p)\}_{k=\overline{1,\infty}}^{m=\overline{1,M}}$  and is determined by the condition $\|g\|^2=\sum _{p=0}^{\infty}
 \|g(p)\|^2=P_\zeta$,
$\vec  \varepsilon(s)=\{\varepsilon_{m}(s)\}_{m=1}^M$ is a vector valued stationary stochastic sequence with orthogonal values.
\end{thm}

\subsection{Classical extrapolation problem}

Let the spectral density $f(\lambda)$ satisfy the \emph{minimality condition}
\be
 \ip \text{Tr}\left[ \frac{\lambda^{2d}}{|1-e^{i\lambda\tau}|^{2d}}f^{-1}(\lambda)\right]
 d\lambda<\infty.
\label{umova11_e_c}
\ee
This is the necessary and sufficient condition under which the mean square errors of the estimates of the functionals $A\vec\xi$ and $A_N\vec\xi$ are not equal to $0$.

Consider the Hilbert space $H=L_2(\Omega,\mathcal{F},\mt P)$ of random variables $\gamma$ with zero mean value, $\mt E\gamma=0$, finite variance, $\mt E|\gamma|^2<\infty$, and endowed with the inner product $(\gamma_1;\gamma_2)=\mt E\gamma_1\overline{\gamma_2}$. With the generated increment sequence $\{\xi^{(d)}_j\}$, we can associate  the closed linear subspace  $H^{0-}(\vec\xi^{(d)})=\overline{span}\{\xi^{(d)}_{kj}:k=1,2,\ldots;j=-1,-2,-3,\ldots\}$.

Consider also  the closed linear subspace  \[L_2^{0-}(f)=\overline{span}\{ e^{i\lambda j}(1-e^{-i\lambda \tau})^d (i\lambda)^{-d} {\delta}_{n}:n=1,2,\ldots,\infty, \, j \leqslant -1\}, \quad {\delta}_n=\{\delta_{kn}\}_{k=1}^{\infty},
\]
of the Hilbert space $L_2(f)$ of vector-valued functions, endowed with the inner product $\langle g_1;g_2\rangle=\ip (g_1(\lambda))^{\top}f(\lambda)\overline{g_2(\lambda)}d\lambda$.
Here $\delta_{kn}$ are Kronecker symbols.
The relation
\begin{equation}\label{spectr_zobr_zeta}
 \xi^{(d)}_{kj}=\ip e^{i\lambda
 j}(1-e^{-i\lambda\tau})^d\dfrac{1}{(i\lambda)^{d}}dZ_k(\lambda), \quad k=1,2,\dots,
\end{equation}
implies  a map between  the subspaces spaces
$H^{0-}(\vec\xi^{(d)})$ and  $L_2^{0-}(f)$. Here $Z_k(\lambda)$ are the components of a vector-valued stochastic process $\vec{Z}_{\xi^{(d)}}(\lambda)=\{Z_{ k}(\lambda)\}_{k=1}^{\infty}$ with uncorrelated increments on $[-\pi,\pi)$   connected with the spectral function $F(\lambda)$ by
the relation
\[
 \mt E(Z_{k}(\lambda_2)-Z_{k}(\lambda_1))(\overline{ Z_{n}(\lambda_2)-Z_{n}(\lambda_1)})
 =F_{kn}(\lambda_2)-F_{kn}(\lambda_1)\]
for
 \[  -\pi\leq \lambda_1<\lambda_2<\pi,\quad k,n=1,2,\dots,\infty.
\]

Define the Fourier coefficients of the function $\lambda^{2d} |1-e^{i\lambda\tau}|^{-2d}f^{-1}(\lambda)$
\[
 F^{\tau}_{j,l}=\frac{1}{2\pi}\ip
 e^{-i\lambda(j-l)}\frac{\lambda^{2d}}{|1-e^{i\lambda\tau}|^{2d}}f^{-1}(\lambda)d\lambda,\quad j,l\geq0.
\]

Define  the vector
\[
\vec{a}_j=(a_{kj},k=1,2,\dots)^{\top}=
(a_{1j},a_{3j},a_{2j},\dots,a_{2k+1,j},a_{2k,j},\dots)^{\top}.
\]
Then  Lemma \ref{lema predst A_cont} implies
\be \label{a_b_relation}
\vec{b}_j=\sum_{m=j}^{\infty}\mt{diag}_{\infty}(d_{\tau}(m-j))\vec{a}_m=(D^{\tau}{\me a})_j,  \quad j=0,1,2,\dots,
\ee
where
$\me a=((\vec{a}_0)^{\top},(\vec{a}_1)^{\top},(\vec{a}_2)^{\top}, \ldots)^{\top}$, the coefficients $\{d_{\tau}(k):k\geq0\}$ are determined by the relationship
\[
 \sum_{k=0}^{\infty}d_{\tau}(k)x^k=\left(\sum_{j=0}^{\infty}x^{\tau j}\right)^d,\]
$D^{\tau}$ is a linear transformation  determined by a matrix with infinite dimension matrix entries
 $D^{\tau}(k,j), k,j=0,1,2,\dots$ such that $D^{\tau}(k,j)=\mt{diag}_{\infty}(d_{\tau}(j-k))$ if $0\leq k\leq j $ and $D^{\tau}(k,j)=\mt{diag}_{\infty}(0)$ for $0\leq j<k$; $\mt{diag}_{\infty}(x)$ denotes an infinite dimension diagonal matrix with the entry $x$ on its diagonal.

\begin{thm}
\label{thm_est_A}
Let a stochastic
process $x(t)$, $t\in \mr R$ with a periodically stationary  increments determine a
generated stationary  $d$th increment sequence
$\vec{\xi}^{(d)}_m$ with the spectral density matrix $f(\lambda)=\{f_{ij}(\lambda)\}_{i,j=1}^{\infty}$ satisfying  minimality condition (\ref{umova11_e_c}).
Let the coefficients $\vec {a}_j, j\geqslant 0$ generated by the function $a(t)$, $t\geq 0$, satisfy conditions  (\ref{ant}) taking into account (\ref{a_b_relation}).
The optimal linear estimate $\widehat{A}\xi$ of the functional $A\xi$ based on observations of the process
$\xi(t)$ on the interval $t<0$ is calculated by the formula
\be \label{otsinka A_e_c}
 \widehat{A}\xi=\ip
(\vec{h}_{\tau}(\lambda))^{\top}d\vec{Z}_{\xi^{(d)}}(\lambda)-\int_{-\tau T d}^{0}v^{\tau}(t)\xi(t)dt,
\ee
 where
$\vec{h}_{\tau}(\lambda)=\{h_{k}(\lambda)\}_{k=1}^{\infty}$ is the spectral characteristic of the estimate $\widehat{A}\vec\xi$.
The spectral characteristic
$\vec h_{\tau}(\lambda)$ and the value of the mean square error $\Delta(f;\widehat{A}\xi)$ of the optimal estimate $\widehat{A}\xi$ are calculated by the formulas
\begin{equation}\label{spectr A_e_c}
 (\vec{h}_{\tau}(\lambda))^{\top}=(\vec{B}_{\tau}(e^{i\lambda}))^{\top}
\frac{(1-e^{-i\lambda\tau})^d}{(i\lambda)^d}-
\frac{(-i\lambda)^{n}
}
{(1-e^{i\lambda\tau})^d}(
\vec{C}_{\tau}(e^{i \lambda}))^{\top}f^{-1}(\lambda).
\end{equation}
 and
\begin{eqnarray}
\nonumber
\Delta\ld(f;\widehat{A}\xi\rd)
&=& \Delta\ld(f;\widehat{B}\xi\,\rd)= \mt E\ld|B\xi-\widehat{B}\xi\,\rd|^2
\\
\nonumber
&=&\frac{1}{2\pi}\int_{-\pi}^{\pi}\frac{(-i\lambda)^d}{(1-e^{i\lambda\tau})^d}
(\vec{C}_{\tau}(e^{i \lambda}))^{\top}f(\lambda)\overline{
\vec{C}_{\tau}(e^{i \lambda})}
\frac{(i\lambda)^d}
{(1-e^{-i\lambda\tau})^d}d\lambda\\
&=&\ld\langle D^{\tau}\me a,\me F_{\tau}^{-1}D^{\tau}\me a\rd\rangle.\label{pohybka}
\end{eqnarray}
 respectively, where
\[
 \vec B_{\tau}(e^{i\lambda})=\sum_{j=0}^{\infty}\vec{b}^{\tau}_je^{i\lambda j},\quad \vec{C}_{\tau}(e^{i \lambda})=\sum_{j=0}^{\infty}(\me F_{\tau}^{-1}D^{\tau}\me a )_j e^{i\lambda j},\]
  $\me F_{\tau}$ is a linear operator in the space $\ell_2$ which is determined by a matrix with the infinite matrix entries $(\me F_{\tau})_{j,l}=F_{j,l}^{\tau}$, $j,l\geq0$.
\end{thm}

\begin{proof}
 Representation (\ref{otsinka A_e_c}) for the estimate $\widehat{A}\xi$ comes from relation (\ref{mainformula}).
The functional $B\vec{\xi}$ allows the spectral representation
\[
B\vec{\xi}=\int_{-\pi}^{\pi}\left(\vec{B}_{\tau}(e^{i\lambda})\right)^{\top} \frac{(1-e^{-i\tau \lambda})^d}{(i\lambda)^d}d\vec{Z}_{\xi^{(d)}}(\lambda).
 \]
Thus, the spectral characteristic $\vec{h}_{\tau}(\lambda)$ is characterized by the following conditions, which determine a projection of the element
$\vec B_{\tau}(e^{i\lambda })\frac{(1-e^{-i\lambda\tau})^d }{(i\lambda)^d}$ on the subspace $L_2^{0-}(f)$:
\begin{equation}
\label{umova1}
\vec{h}_{\tau}(\lambda)\in L_2^{0-}\ld(f\rd),
\end{equation}
\begin{equation}
\label{umova2}
\ld(\vec{B}_{\tau}(e^{i\lambda})\frac{(1-e^{-i\lambda\tau})^d }{(i\lambda)^d}-\vec{h}_{\tau}(\lambda)\rd)\perp
L_2^{0-}\ld(f\rd).
\end{equation}

From the condition \eqref{umova2}, we obtain the following
relation which holds true for all $  j\leqslant -1$
\begin{gather}
\label{umova2_2}
 \int_{-\pi}^{\pi}\left(\vec{B}_{\tau}(e^{i\lambda}) \frac{(1-e^{-i\tau \lambda})^d}{(i\lambda)^d}-\vec{h}_{\tau}(\lambda)\right)^{\top} f(\lambda)e^{-ij\lambda}\frac{(1-e^{i\tau \lambda})^d}{(-i\lambda)^d} d\lambda=0.
\end{gather}

Thus, the spectral characteristic of the estimate
$\widehat{B}\vec\xi$ can be represented in the form
\[ (\vec{h}_{\tau}(\lambda))^{\top}=(\vec{B}_{\tau}(e^{i\lambda}))^{\top}
\frac{(1-e^{-i\lambda \tau})^d}{(i\lambda)^d}-
\frac{(-i\lambda)^{n}(\vec{C}_{\tau}(e^{i\lambda}))^{\top}}{(1-e^{i\lambda \tau})^d}f^{-1}(\lambda),\]
where\[\vec{C}_{\tau}(e^{i \lambda})=\sum_{j=0}^{\infty}\vec{c}^{\,\tau}_je^{ij\lambda},\]
coefficients $\vec{c}_j^{\,\tau}=\{c_{kj}^{\tau}\}_{k=1}^{\infty}, j\geqslant 0$  are unknown and to be found.

Condition \eqref{umova1}  implies the equestions
\begin{equation} \label{eq_C}
\int_{-\pi}^{\pi} \left[(\vec{B}_{\tau}(e^{i\lambda}))^{\top}-
\frac{\lambda^{2d}(\vec{C}_{\tau}(e^{i\lambda}))^{\top}}{(1-e^{-i\lambda \tau})^d (1-e^{i\lambda \tau})^d}f^{-1}(\lambda)\right]e^{-ij\lambda}d\lambda=0,  j\geqslant 0.
\end{equation}

\noindent Making use of the introduced Fourier coefficients $F^{\tau}_{j,l}$, relation \eqref{eq_C} can be presented as a system of linear equations
\[
 \vec{b}^{\tau}_l=\sum_{j=0}^{\infty}F_{l,j}^{\tau}\,\,{\vec c}^{\,\tau}_j,\quad l\geq0,
 \]
 allowing the matrix form representation
\be\label{meq-e}
 D^{\tau}\me a=\me F_{\tau}\me c^{\tau},
\ee
where $\me c^{\tau}=(({\vec c}^{\,\tau}_0)^{\top},({\vec c}^{\,\tau}_1)^{\top},({\vec c}^{\,\tau}_2)^{\top}, \ldots)^{\top}$.

Note that the operator $\me F_{\tau}$ is invertible. The   projection problem for the element ${B}\vec\xi \in H$  on the closed convex set $H^{0-}(\vec\xi^{(d)}_{\tau})$ has a unique solution for each non-zero vector coefficients $\{\vec{a}_0,\vec{a}_1,\vec{a}_2, \ldots\}$ under conditions
(\ref{ant}). Therefore, equation (\ref{meq-e}) has a unique solution for each vector $\me b^{\tau}=D^{\tau} \me a$, which implies existence of the inverse operator $\me F^{-1}_{\tau}$.
Now we can conclude, that the coefficients ${\vec c}^{\,\tau}_j$, $j\geq0$,  can be calculated as
\begin{equation}
 \vec c^{\,\tau}_j=\ld(\me F_{\tau}^{-1}D^{\tau}\me a\rd)_j,\quad j\geq 0,
 \end{equation}
 where $\ld(\me F_{\tau}^{-1}D^{\tau}\me a\rd)_j$, $j\geq 0$, is the
 $j$th  infinite dimension vector element of the vector $\me F_{\tau}^{-1}D^{\tau}\me a$.

The derived expressions   and the coefficients ${\vec c}^{\,\tau}_j$, $j\geq0$, prove  formulas (\ref{spectr A_e_c}) for the   spectral characteristic $\vec{h}_{\tau}(\lambda)$
 and (\ref{pohybka}) for the  value of the mean square error of the estimate $\widehat{A}\xi$.
\end{proof}

Consider the extrapolation problem for the functional $A_{NT}\xi$. We can apply the results of Theorem \ref{thm_est_A} by putting $(\vec{a}_j)^{\top}=0$ for $j>N$. Corollary \ref{nas predst A_T_cont} implies
\[
\vec{b}^d_j=\sum_{m=j}^N\mt{diag}_{\infty}(d_{\tau}(m-j))\vec{a}(m)=(D^{\tau}_{N}{\me a}_{N})_j, \, j=0,1,\dots,N,
\]
 where $D^{\tau}_{N}$ is the  linear transformation  determined by an infinite matrix with the entries $(D^{\tau}_{N})(k,j)=\mt{diag}_{\infty}(d_{\tau}(j-k))$ if
$0\leq k\leq j\leq N$, and $(D^{\tau}_{N})(k,j)=0$ if $j<k$ or $j,k>N$; $\me a_N=((\vec{a}_0)^{\top},(\vec{a}_1)^{\top}, \ldots,(\vec{a}_N)^{\top},\vec 0 \ldots)^{\top}$.

\begin{thm}
\label{thm_est_A_N}
Let a stochastic
process $\xi(t)$, $t\in \mr R$ with a periodically stationary  increments determine a
generated stationary  $d$th increment sequence
$\vec{\xi}^{(d)}_m$ with the spectral density matrix $f(\lambda)=\{f_{ij}(\lambda)\}_{i,j=1}^{\infty}$ satisfying  minimality condition (\ref{umova11_e_c}).
The optimal linear estimate $\widehat{A}_{NT}\xi$ of the functional $A_{NT}\xi$ based on observations of the process
$\xi(t)$ on the interval $t<0$ is calculated by the  formula
\begin{equation} \label{estimation_A_N}
\widehat{A}_{NT}{\xi}=\int_{-\pi}^{\pi}
(\vec{h}_{\tau,N}(\lambda))^{\top}d\vec{Z}_{\xi^{(d)}}(\lambda) - \int_{-\tau T d}^{0}v^{\tau,N}(t)\xi(t)dt,
\end{equation}
The spectral characteristic $\vec{h}_{\tau,N}(\lambda)=\{h_{\tau,N,k}(\lambda)\}_{k=1}^{\infty}$ and the value of the mean square error $\Delta(f;\widehat{A}_N \xi)$
are calculated by the formulas  \begin{multline} \label{est_h_N}
(\vec{h}_{\tau,N}(\lambda))^{\top}=(\vec{B}_{\tau,N}(e^{i\lambda}))^{\top}
\frac{(1-e^{-i\lambda \tau})^d}{(i\lambda)^d}-
\frac{(-i\lambda)^{n}
}
{(1-e^{i\lambda \tau})^d}(
\vec{C}_{\tau,N}(e^{i \lambda}))^{\top}f^{-1}(\lambda).
\end{multline}
and
\begin{eqnarray}
\nonumber
\Delta (f, \widehat{A}_{NT}{\xi})&=&\Delta (f, \widehat{B}_{NT}{\xi})={\mathsf E}|B_{NT}{\xi}-\widehat{B}_{NT}{\xi}|^2 \\
\nonumber
&=&\frac{1}{2\pi}\int_{-\pi}^{\pi}\frac{(-i\lambda)^d
}
{(1-e^{i\lambda\tau})^d}(
\vec{C}_{\tau,N}(e^{i \lambda}))^{\top}f(\lambda)\overline{
\vec{C}_{\tau,N}(e^{i \lambda})}
\frac{(i\lambda)^d}
{(1-e^{-i\lambda\tau})^d}d\lambda\\
\label{pohybka_N}
&=&\ld\langle D^{\tau}_{N}{\me a}_N,\me F_{\tau,N}^{-1}D^{\tau}_{N}{\me a}_N\rd\rangle.
\end{eqnarray}
respectively, where
\[
\vec B_{\tau,N}(e^{i\lambda })=\sum_{j=0}^{N}(D^{\tau}_N{\me a}_N)_je^{i\lambda j},\quad
\vec{C}_{\tau,N}(e^{i \lambda})=\sum_{j=0}^{\infty}(\me F_{\tau,N}^{-1}D^{\tau}_{N}{\me a}_N )_j e^{i\lambda j}
\]
$\me F_{\tau,N}$ is a linear operator in the space $\ell_2$
which is determined by a matrix with the infinite  matrix entries
 $(\me F_{\tau,N})_{l, m} =F^{\tau}_{l,m}$, $l\geq0$, $0\leq m\leq N$, and $(\me F_{\tau,N})_{l, m} =0$, $l\geq0$, $m>N$.
\end{thm}

\section{Minimax (robust) method of extrapolation}

The values of the mean square errors and the spectral characteristics of the optimal estimates
of the functionals ${A}\vec\xi$ and ${A}_{NT}\vec\xi$
depending on the unobserved values of a stochastic process ${\xi}(t)$, which determine a periodically stationary stochastic    increment process
${\xi}^{(d)}(t,\tau T)$ or the corresponding to it infinite dimension vector stationary increment sequence
  $\{\vec\xi^{(d)}_j=(\xi^{(d)}_{kj}, k=1,2,\dots)^{\top}$, $
j\in\mathbb Z\}$,
with the spectral density matrix $f(\lambda)$,
 can be calculated by formulas
(\ref{spectr A_e_c}), (\ref{pohybka}) and
\eqref{est_h_N}, \eqref{pohybka_N}
respectively, provided the spectral density
$f(\lambda)$ of the stochastic process $\vec\xi(t)$ is exactly known.
If the spectral density is not known, but a set $\md D$ of admissible spectral densities is defined,
the minimax method of estimation of the functionals depending on unobserved values of stochastic sequences with stationary increments may be applied.
This method consists in finding an estimate that minimizes
the maximal values of the mean square errors for all spectral densities
from a given class $\mathcal D$ of admissible spectral densities
simultaneously. The described method is formalized by the two following definitions \cite{Moklyachuk:2008}.

\begin{ozn}
For a given class of spectral densities $\mathcal{D}$ the
spectral density $f_0(\lambda)\in\mathcal{D}$ is called the least
favourable in $\mathcal{D}$ for the optimal linear estimation of the
functional $A\vec \xi$ if the following relation holds true:
\[\Delta(f_0)=\Delta(h_{\tau}(f_0);f_0)=\max_{f\in\mathcal{D}}\Delta(h_{\tau}(f);f).
\]
\end{ozn}

\begin{ozn}
 For a given class of spectral densities $\mathcal{D}$
the spectral characteristic $h^0(\lambda)$ of the optimal linear estimate of the functional
$A \xi$ is called minimax-robust if the following conditions are satisfied
\[h^0(\lambda)\in
H_{\mathcal{D}}=\bigcap_{f\in\mathcal{D}}L_2^{0-}(f),
\]
\[
\min_{h\in H_{\mathcal{D}}}\max_{f\in
\mathcal{D}}\Delta(h;f)=\sup_{f\in\mathcal{D}}\Delta(h^0;f).
\]
\end{ozn}

Taking into account the introduced definitions and the derived relations we can verify that the following lemmas hold true.

\begin{lema}
The spectral density $f_0(\lambda)\in\mathcal{D}$ satisfying the minimality condition (\ref{umova11_e_c}) is the least favourable density in the class $\mathcal{D}$ for the optimal linear extrapolation of the functional $A\xi$ based on observations of the process $\xi(t)$ on the interval $t<0$ if the operator $\me F_{\tau}^0$ defined by
the Fourier coefficients of the function
\be\label{functions_for_lemma_e_c}
 \lambda^{2d}|1-e^{i\lambda\tau}|^{-2d}f_0^{-1}(\lambda),
 \ee
determines a solution to the constrained optimization problem
\be
 \max_{f\in \mathcal{D}}\ld(\ld\langle D^{\tau}\me a,\me F_{\tau}^{-1}D^{\tau}\me a\rd\rangle\rd)= \ld\langle D^{\tau}\me a,(\me F^0_{\tau})^{-1}D^{\tau}\me a\rd\rangle.
\label{minimax1_e_c}
\ee
The minimax spectral characteristic $h^0=h_{\tau}(f^0)$ is calculated by formula (\ref{spectr A_e_c}) if
$h_{\tau}(f^0)\in H_{\mathcal{D}}$.
\end{lema}

For more detailed analysis of properties of the least favorable spectral densities and the minimax-robust spectral characteristics, we observe that the minimax spectral characteristic $h^0$ and the least favourable spectral density $f_0$ form a saddle
point of the function $\Delta(h;f)$ on the set
$H_{\mathcal{D}}\times\mathcal{D}$ \cite{Moklyachuk:2008}.
The saddle point inequalities
\[
 \Delta(h;f_0)\geq\Delta(h^0;f_0)\geq\Delta(h^0;f)\quad\forall f\in
 \mathcal{D},\forall h\in H_{\mathcal{D}}\]
hold true if $h^0=h_{\tau}(f_0)$,
$h_{\tau}(f_0)\in H_{\mathcal{D}}$ and $f_0$ is a solution of the constrained optimization problem
\be
 \widetilde{\Delta}(f)=-\Delta(h_{\tau}(f_0);f)\to
 \inf,\quad f\in \mathcal{D},\label{zad_um_extr_e_c}
 \ee
where the functional $\Delta(h_{\tau}(f_0);f)$ is calculated by the formula
\[
 \Delta(h_{\tau}(f_0);f)
=\frac{1}{2\pi}\int_{-\pi}^{\pi}
\frac{(-i\lambda)^d
}
{(1-e^{i\lambda\tau})^d}(\vec{C}_{\tau}^0(e^{i \lambda})
)^{\top}
f_0^{-1}(\lambda)
f(\lambda)
 f_0^{-1}(\lambda)\overline{
\vec{C}_{\tau}^0(e^{i \lambda})}
\frac{(i\lambda)^d}
{(1-e^{-i\lambda\tau})^d}d\lambda
\]
where
\[
\vec{C}_{\tau}^0(e^{i \lambda})=\sum_{j=0}^{\infty}((\me F^0_{\tau})^{-1}D^{\tau}\me a )_j e^{i\lambda j}
\].

The constrained optimization problem (\ref{zad_um_extr_e_c}) is equivalent to the unconstrained optimization problem
\[
 \Delta_{\mathcal{D}}(f)=\widetilde{\Delta}(f)+ \delta(f|\mathcal{D})\to\inf,\]
where $\delta(f|\mathcal{D})$ is the indicator function of the set
$\mathcal{D}$, namely $\delta(f|\mathcal{D})=0$ if $f\in \mathcal{D}$ and $\delta(f|\mathcal{D})=+\infty$ if $f\notin \mathcal{D}$.
 A solution $f_0$ of the unconstrained optimization problem is characterized by the condition
 $0\in\partial\Delta_{\mathcal{D}}(f_0)$, which is the necessary and sufficient condition under which the point $f_0$ belongs to the set of minimums of the convex functional $\Delta_{\mathcal{D}}(f)$ \cite{Franke1984,Moklyachuk:2008,Rockafellar}.
 This condition makes it possible to find the least favourable spectral densities in some special classes of spectral densities $D$.

The form of the functional $\widetilde{\Delta}(f)$ allows us to apply the Lagrange method of indefinite
multipliers for investigating the constrained optimization problem (\ref{zad_um_extr_e_c}).

\subsection{Least favorable spectral density in classes $\md D_0$}

Consider the   extrapolation problem for the functional $A{\xi}$
 which depends on unobserved values of a process $\xi(t)$ with periodically stationary increments based on observations of the process at points $t<0$ under the condition that the sets of admissible spectral densities $\md D_0^k,k=1,2,3,4$ are defined as follows:
\begin{eqnarray*}
\md D_{0}^{1} &=& \bigg\{f(\lambda )\left|\frac{1}{2\pi} \int
_{-\pi}^{\pi}
\frac{|1-e^{i\lambda\tau}|^{2d}}{\lambda^{2d}}
f(\lambda )d\lambda  =P\right.\bigg\},
\\
\md D_{0}^{2} &=&\bigg\{f(\lambda )\left|\frac{1}{2\pi }
\int _{-\pi }^{\pi}
\frac{|1-e^{i\lambda\tau}|^{2d}}{\lambda^{2d}}
{\rm{Tr}}\,[ f(\lambda )]d\lambda =p\right.\bigg\},
\\
\md D_{0}^{3} &=&\bigg\{f(\lambda )\left|\frac{1}{2\pi }
\int _{-\pi}^{\pi}
\frac{|1-e^{i\lambda\tau}|^{2d}}{|\lambda|^{2d}}
f_{kk} (\lambda )d\lambda =p_{k}, k=1,2,3,\ldots\right.\bigg\},
\\
\md D_{0}^{4} &=&\bigg\{f(\lambda )\left|\frac{1}{2\pi} \int _{-\pi}^{\pi}
\frac{|1-e^{i\lambda\tau}|^{2d}}{|\lambda|^{2d}}
\left\langle B_{1} ,f(\lambda )\right\rangle d\lambda  =p\right.\bigg\},
\end{eqnarray*}
where  $p, p_k, k=1,2,3,\ldots$ are given numbers, $P, B_1,$ are given positive-definite Hermitian matrices.

From the condition $0\in\partial\Delta_{\mathcal{D}}(f_0)$
we find the following equations which determine the least favourable spectral densities for these given sets of admissible spectral densities.

For the   set $\md D_{0}^{1}$ of admissible spectral densities, we have equation
\be  \label{eq_4_1}
\left(
\vec{C}_{\tau}^0(e^{i \lambda})
\right)
\left(
\vec{C}_{\tau}^0(e^{i \lambda})
\right)^{*}=\left(\frac{|1-e^{i\lambda\tau}|^{2d}}{\lambda^{2d}} f_{0} (\lambda )\right)
\vec{\alpha}\cdot \vec{\alpha}^{*}\left(\frac{|1-e^{i\lambda\tau}|^{2d}}{\lambda^{2d}} f_{0} (\lambda )\right),
\ee
where $\vec{\alpha}$ is a vector of Lagrange multipliers.

For the   set $\md D_{0}^{2}$ of admissible spectral densities, we have equation
\begin{equation}  \label{eq_4_2}
\left(
\vec{C}_{\tau}^0(e^{i \lambda})
\right)
\left(\vec{C}_{\tau}^0(e^{i \lambda})
\right)^{*}=
\alpha^{2} \left(\frac{|1-e^{i\lambda\tau}|^{2d}}{\lambda^{2d}} f_{0} (\lambda )\right)^{2},
\end{equation}
where $\alpha^{2}$ is a Lagrange multiplier.

For the   set $\md D_{0}^{3}$ of admissible spectral densities, we have equation
\be   \label{eq_4_3}
\left(
\vec{C}_{\tau}^0(e^{i \lambda})
\right)
\left(\vec{C}_{\tau}^0(e^{i \lambda})
\right)^{*}
=
\left(\frac{|1-e^{i\lambda\tau}|^{2d}}{\lambda^{2d}} f_{0} (\lambda )\right)\left\{\alpha _{k}^{2} \delta _{kl} \right\}_{k,l=1}^{\infty} \left(\frac{|1-e^{i\lambda\tau}|^{2d}}{\lambda^{2d}} f_{0} (\lambda )\right),
\ee
where  $\alpha _{k}^{2}$ are Lagrange multipliers, $\delta _{kl}$ are Kronecker symbols.

For the   set $\md D_{0}^{4}$ of admissible spectral densities, we have equation
\be    \label{eq_4_4}
\left(
\vec{C}_{\tau}^0(e^{i \lambda})
\right)
\left(\vec{C}_{\tau}^0(e^{i \lambda})
\right)^{*}
=
\alpha^{2} \left(\frac{|1-e^{i\lambda\tau}|^{2d}}{\lambda^{2d}} f_{0} (\lambda )\right)B_{1}^{\top} \left(\frac{|1-e^{i\lambda\tau}|^{2d}}{\lambda^{2d}} f_{0} (\lambda )\right),
\ee
where $\alpha^{2}$ is a Lagrange multiplier.

The following theorem  holds true.

\begin{thm}
Let   minimality condition (\ref{umova11_e_c}) hold true. The least favorable spectral densities $f_{0}(\lambda)$ in the classes $ \md  D_0^{k}$, $k=1,2,3,4$, for the optimal linear extrapolation of the functional  $A{\xi}$ from observations of the process ${\xi}(t)$ on the interval $t<0$  are determined by the equations
\eqref{eq_4_1}, \eqref{eq_4_2}, \eqref{eq_4_3}, \eqref{eq_4_4},
the constrained optimization problem (\ref{minimax1_e_c}) and restrictions  on densities from the corresponding classes $ \md  D_0^{k}$, $k=1,2,3,4$.  The minimax-robust spectral characteristic of the optimal estimate of the functional $A{\xi}$ is determined by the formula (\ref{spectr A_e_c}).
\end{thm}

\subsection{Least favorable spectral density in classes $\md D_{1\delta}$}\label{set1}

Consider the extrapolation problem for the functional $A{\xi}$
 which depends on unobserved values of the process $\xi(t)$ with periodically stationary increments based on its observations   at points $t<0$ under the condition that the sets of admissible spectral densities ${D_{1\delta}^{k}},k=1,2,3,4$ are defined as follows:
\begin{eqnarray*}
\md D_{1\delta}^{1}&=&\left\{f(\lambda )\biggl|\frac{1}{2\pi} \int_{-\pi}^{\pi}
\left|f_{ij} (\lambda )-f_{ij}^{1} (\lambda)\right|d\lambda  \le \delta_{i}^j, i,j=\overline{1,\infty}\right\}.
\\
\md D_{1\delta}^{2}&=&\left\{f(\lambda )\biggl|\frac{1}{2\pi} \int_{-\pi}^{\pi}
\left|{\rm{Tr}}(f(\lambda )-f_{1} (\lambda))\right|d\lambda \le \delta\right\};
\\
\md D_{1\delta}^{3}&=&\left\{f(\lambda )\biggl|\frac{1}{2\pi } \int_{-\pi}^{\pi}
\left|f_{kk} (\lambda )-f_{kk}^{1} (\lambda)\right|d\lambda  \le \delta_{k}, k=\overline{1,\infty}\right\};
\\
\md D_{1\delta}^{4}&=&\left\{f(\lambda )\biggl|\frac{1}{2\pi } \int_{-\pi}^{\pi}
\left|\left\langle B_{2} ,f(\lambda )-f_{1}(\lambda )\right\rangle \right|d\lambda  \le \delta\right\};
\end{eqnarray*}

\noindent
Here  $f_{1} ( \lambda )=\{f_{ij}^{1} ( \lambda )\}_{i,j=1}^{\infty}$ is a fixed spectral density,  $B_2$ is a given positive-definite Hermitian matrix,
$\delta,\delta_{k},k=\overline{1,\infty}$, $\delta_{i}^{j}, i,j=\overline{1,\infty}$, are given numbers.

From the condition $0\in\partial\Delta_{\mathcal{D}}(f_0)$,
we find the following equations which determine the least favourable spectral densities for these given sets of admissible spectral densities.

For the  set $\md D_{1\delta}^{1}$ of admissible spectral densities,   we have equations
\be   \label{eq_5_1g}
\left(
\vec{C}_{\tau}^0(e^{i \lambda})
\right)
\left(\vec{C}_{\tau}^0(e^{i \lambda})
\right)^{*}
=
\left(\frac{|1-e^{i\lambda\tau}|^{2d}}{\lambda^{2d}}f_0(\lambda)\right)
\left \{ \beta_{ij}\gamma_{ij} ( \lambda ) \right \}_{i,j=1}^{\infty}
\left(\frac{|1-e^{i\lambda\tau}|^{2d}}{\lambda^{2d}}f_0(\lambda)\right),
\ee
\begin{equation} \label{eq_5_1c}
\frac{1}{2 \pi} \int_{- \pi}^{ \pi} \left|f_{ij}^0(\lambda)-f_{ij}^{1}( \lambda ) \right|d\lambda = \delta_{i}^{j},
\end{equation}
where  $ \beta_{ij}$ are Lagrange multipliers,  functions $\left| \gamma_{ij} ( \lambda ) \right| \le 1$ and
\[
\gamma_{ij} ( \lambda )= \frac{f_{ij}^0  ( \lambda )-f_{ij}^{1} (\lambda )}{ \left|f_{ij}^0  ( \lambda )-f_{ij}^{1}(\lambda) \right|}: \; f_{ij}^0  ( \lambda )-f_{ij}^{1} ( \lambda ) \ne 0, \; i,j= \overline{1,\infty}.
\]

For the  set $\md D_{1\delta}^{2}$ of admissible spectral densities,   we have equations
\be  \label{eq_5_2g}
\left(
\vec{C}_{\tau}^0(e^{i \lambda})
\right)
\left(\vec{C}_{\tau}^0(e^{i \lambda})
\right)^{*}
=
\beta^{2} \gamma_2( \lambda )\left(\frac{|1-e^{i\lambda\tau}|^{2d}}{\lambda^{2d}}f_0(\lambda)\right)^2,
\ee
\begin{equation} \label{eq_5_2c}
\frac{1}{2 \pi} \int_{-\pi}^{ \pi}
\left|{\mathrm{Tr}}\, (f_0( \lambda )-f_{1}(\lambda )) \right|d\lambda =\delta,
\end{equation}

\noindent where   $ \beta^{2}$ are Lagrange multipliers,   the function $\left| \gamma_2( \lambda ) \right| \le 1$ and
\[\gamma_2( \lambda )={ \mathrm{sign}}\; ({\mathrm{Tr}}\, (f_{0} ( \lambda )-f_{1} ( \lambda ))): \; {\mathrm{Tr}}\, (f_{0} ( \lambda )-f_{1} ( \lambda )) \ne 0.\]

For the  set $\md D_{1\delta}^{3}$ of admissible spectral densities,   we have equations
\be    \label{eq_5_3g}
\left(
\vec{C}_{\tau}^0(e^{i \lambda})
\right)
\left(\vec{C}_{\tau}^0(e^{i \lambda})
\right)^{*}
=
\left(\frac{|1-e^{i\lambda\tau}|^{2d}}{\lambda^{2d}}f_0(\lambda)\right)
\left \{ \beta_{k}^{2} \gamma^2_{k} ( \lambda ) \delta_{kl} \right \}_{k,l=1}^{\infty}
\left(\frac{|1-e^{i\lambda\tau}|^{2d}}{\lambda^{2d}}f_0(\lambda)\right),
\ee
\begin{equation} \label{eq_5_3c}
\frac{1}{2 \pi} \int_{- \pi}^{ \pi}  \left|f_{kk}^0 ( \lambda)-f_{kk}^{1} ( \lambda ) \right| d\lambda =\delta_{k},
\end{equation}

\noindent where   $\beta_{k}^{2}$ are Lagrange multipliers, $\delta _{kl}$ are Kronecker symbols,  functions $\left| \gamma^2_{k} ( \lambda ) \right| \le 1$ and
\[\gamma_{k}^2( \lambda )={ \mathrm{sign}}\;(f_{kk}^0( \lambda)-f_{kk}^{1} ( \lambda )): \; f_{kk}^0( \lambda )-f_{kk}^{1}(\lambda ) \ne 0, \; k= \overline{1,\infty}.\]

For the  set $\md D_{1\delta}^{4}$ of admissible spectral densities,   we have equations
\be  \label{eq_5_4g}
\left(
\vec{C}_{\tau}^0(e^{i \lambda})
\right)
\left(\vec{C}_{\tau}^0(e^{i \lambda})
\right)^{*}
=
\beta^{2} \gamma_2'( \lambda )
\left(\frac{|1-e^{i\lambda\tau}|^{2d}}{\lambda^{2d}}f_0(\lambda)\right)
B_{2}^{ \top}
\left(\frac{|1-e^{i\lambda\tau}|^{2d}}{\lambda^{2d}}f_0(\lambda)\right),
\ee
\begin{equation} \label{eq_5_4c}
\frac{1}{2 \pi} \int_{- \pi}^{ \pi}  \left| \left \langle B_{2}, f_0( \lambda )-f_{1} ( \lambda ) \right \rangle \right|d\lambda
= \delta,
\end{equation}
where   $\beta^{2}$ are Lagrange multipliers,   function
$\left| \gamma_2' ( \lambda ) \right| \le 1$ and
\[\gamma_2' ( \lambda )={ \mathrm{sign}}\; \left \langle B_{2},f_{0} ( \lambda )-f_{1} ( \lambda ) \right \rangle : \; \left \langle B_{2},f_{0} ( \lambda )-f_{1} ( \lambda ) \right \rangle \ne 0.\]

The derived results are summarized in the  following theorem.

\begin{thm}
Let   minimality condition (\ref{umova11_e_c}) hold true. The least favorable spectral densities $f_{0} ( \lambda )=\{f_{ij}^{0} ( \lambda )\}_{i,j=1}^{\infty}$  in the classes $\md D_{1\delta}^{k},k=1,2,3,4$  for the optimal linear extrapolation of the functional  $A{\xi}$ from observations of the process ${\xi}(t)$ at points  $t<0$  are determined by equations
 \eqref{eq_5_1g} -- \eqref{eq_5_1c},   \eqref{eq_5_2g} -- \eqref{eq_5_2c},  \eqref{eq_5_3g} -- \eqref{eq_5_3c},  \eqref{eq_5_4g} -- \eqref{eq_5_4c},
respectively,
the constrained optimization problem (\ref{minimax1_e_c}) and restrictions  on the densities from the corresponding classes $\md D_{1\delta}^{k},k=1,2,3,4$.  The minimax-robust spectral characteristic of the optimal estimate of the functional $A{\xi}$ is determined by the formula (\ref{spectr A_e_c}).
\end{thm}

\begin{section}{Conclusions}

In this article, we dealt with  continuous time stochastic processes with periodically correlated $d$-th increments.
These stochastic processes form a class of non-stationary stochastic processes
that combine  periodic structure of covariation functions of processes as well as integrating one.

We derived solutions of the problem of estimation of the linear functionals constructed from the  unobserved values of a continuous time stochastic process with periodically correlated $d$-th increments.
Estimates are based on observations of the process at points $t<0$.
We obtained the estimates by representing the process under investigation as a vector-valued sequence with stationary increments.
Based on the  solutions for these type of sequences, we solved the corresponding problem for the defined class of continuous time stochastic processes.
The problem is investigated in the case of spectral certainty, where spectral densities of sequences are exactly known.
In this case  we propose an approach based on the Hilbert space projection method.
We derive formulas for calculating the spectral characteristics and the mean-square errors of the optimal estimates of the functionals.
In the case of spectral uncertainty where the spectral densities are not exactly known while, instead,
some sets of admissible spectral densities are specified, the minimax-robust method is applied.
We propose a representation of the mean square error in the form of a linear
functional in $L_1$ with respect to spectral densities, which allows
us to solve the corresponding constrained optimization problem and
describe the minimax-robust estimates of the functionals. Formulas
that determine the least favorable spectral densities and minimax-robust spectral characteristics of the optimal linear estimates of
the functionals are derived  for a wide list of specific classes
of admissible spectral densities.

\end{section}

\end{document}